%% file: p35.tex
\documentclass[12pt]{article}  
\usepackage{amsmath} 
\usepackage{latexsym}

\usepackage{url}
\usepackage{amssymb} 
\usepackage{exscale} 
\usepackage{graphicx} 
\usepackage{makeidx}

\usepackage{array}
\newcolumntype{C}[1]{>{\centering\let\newline\\\arraybackslash\hspace{0pt}}m{#1}}

\newtheorem{thm}{Theorem}
\newtheorem{cor}[thm]{Corollary}
\newtheorem{con}[thm]{Conjecture}

\newtheorem{lem}[thm]{Lemma}

\newenvironment{proof}[1][Proof]
{\par\noindent{\bf #1.} }{\hspace*{\fill}\nolinebreak{$\Box$}\bigskip\par}

\title{\bf On Some Three-Color\\
Ramsey Numbers for Paths}

\author{
Janusz Dybizba\'{n}ski\footnote{Supported by Polish
National Science Centre contract DEC-2012/05/N/ST6/03063},
Tomasz Dzido\footnote{Supported by Polish National
Science Centre grant 2011/02/A/ST6/00201}\\
\small Institute of Informatics, University of Gda\'{n}sk\small \\[-0.8ex]
\small Wita Stwosza 57, 80-952 Gda\'{n}sk, Poland\\[-0.8ex]
\small \texttt{\{jdybiz,tdz\}@inf.ug.edu.pl}\\
\\[-2ex]
and\\[-2ex]
\\
Stanis\l aw Radziszowski\footnotemark[\value{footnote}]\\
\small Department of Computer Science \\[-0.8ex]
\small Rochester Institute of Technology \\[-0.8ex]
\small Rochester, NY  14623, USA \\[-0.8ex]
\small \texttt{spr@cs.rit.edu}\\[-3ex]
}

\date{}

\begin{document}
\maketitle
\begin{abstract}
\noindent
For graphs $G_1, G_2, G_3$,
the three-color Ramsey number $R(G_1,$ $G_2, G_3)$
is the smallest integer $n$ such that if we arbitrarily color
the edges of the complete graph of order $n$ with $3$ colors,
then it contains a monochromatic copy of $G_i$ in color $i$,
for some $1 \leq i \leq 3$.

First, we prove that the conjectured equality
$R(C_{2n},C_{2n},C_{2n})=4n$, if true, implies that
$R(P_{2n+1},P_{2n+1},P_{2n+1})=4n+1$ for all $n \ge 3$.
We also obtain two new exact values $R(P_8,P_8,P_8)=14$
and $R(P_9,P_9,P_9)=17$,
furthermore we do so without help of computer algorithms.
Our results agree with a formula $R(P_n,P_n,P_n)=2n-2+(n\bmod 2)$
which was proved for sufficiently large $n$ by Gy\'arf\'as,
Ruszink\'{o}, S\'{a}rk\"{o}zy, and Szemer\'{e}di in 2007.
This provides more evidence for the conjecture that
the latter holds for all $n \ge 1$.
\end{abstract}

\section{Definitions}

In this paper all graphs are undirected, finite and
contain neither loops nor multiple edges. Let $G$ be such a graph.
The vertex set of $G$ is denoted by $V(G)$, the edge set of $G$
by $E(G)$, and the number of edges in $G$ by $e(G)$. For any edge
coloring $F$ of a complete graph, $F^{i}$ will denote the graph
induced by the edges of color $i$ in $F$. Let $P_k$ (resp. $C_k$)
be the path (resp. cycle) on $k$ vertices.
The \emph{circumference} $c(G)$ of a graph $G$ is
the length of its longest cycle.

\emph{The Tur{\'a}n number} $T(n,G)$ is the maximum number of
edges in any $n$-vertex graph that does not contain any subgraph
isomorphic to $G$. A graph on $n$ vertices is said to be
\emph{extremal with respect to $G$} if it does not contain
a subgraph isomorphic to $G$ and has exactly $T(n,G)$ edges.

\medskip
For given graphs $G_{1}, G_{2}, ... , G_{k}, k \geq 2$,
the \emph{multicolor Ramsey number} $R(G_{1}, G_{2}, ... , G_{k})$
is the smallest integer $n$ such that if we arbitrarily color
with $k$ colors
the edges of the complete graph of order $n$, $K_n$,
then it contains a monochromatic copy of $G_{i}$ in color $i$,
for some $1 \leq i \leq k$. A coloring of the edges of $K_n$
with $k$ colors is called a
$(G_1, G_2, ..., G_k ; n)$-coloring, if it does not
contain a subgraph isomorphic to $G_i$ in color $i$
for any $1\le i\le k$. In the diagonal cases, when
for all $i$, $1\le i\le k$, $G_i=G$ for some graph $G$,
we will write $R(G_{1}, G_{2}, ... , G_{k})=R_k(G)$.
Finally, we will refer to the first three colors of such Ramsey
colorings as red, blue and green, respectively.

\section{Overview}

In this article we study the values of three-color diagonal
Ramsey numbers for paths. In the case of two color Ramsey numbers,
a well known theorem of Gerencs\'er and Gy\'arf\'as \cite{GeGy}
states that $R(P_n,P_m)= n+\lfloor \frac{m}{2} \rfloor -1$
for $n \geq m \geq 2$.

\medskip
Clearly, we have $R_3(P_1)=1$ and $R_3(P_2)=2$. 
The cases $R_3(P_3)=5$ and $R_3(P_4)=6$ are easy but need
some thought, while the results $R_3(P_5)=9$, $R_3(P_6)=10$
and $R_3(P_7)=13$ already required help of computer algorithms
(see section 6.4.1 of \cite{Rad} for details and references
to these and other related cases). The first open cases are
those of $R_3(P_8)$ and $R_3(P_9)$, which are determined
later in this paper.
All known values agree with a very remarkable result
obtained by Gy\'arf\'as, Ruszink\'{o}, S\'{a}rk\"{o}zy,
and Szemer\'{e}di in 2007 \cite{GRSS} formulated as follows.

\begin{thm}[\cite{GRSS}]
For all sufficiently large $n$, we have
$$
R_3(P_n) = 
\left \{
\begin{array}{ll}
2n - 1 & \textrm{ for odd } n,\\
2n - 2 & \textrm{ for even } n.
\end{array}
\right. \eqno{(1)}
$$
\end{thm}

\medskip
The proof of Theorem 1 is very long and complicated. Our attempts to
extract from it any reasonable bound on how large $n$ should
be for (1) to hold, failed. Actually, Faudree and Schelp \cite{fas},
already in 1975, stated that ``they feel" that (1) holds for all $n$.
They did so when considering more general cases of $R(P_m,P_n,P_k)$
for paths of different lengths. We believe that the diagonal
case deserves the status of a conjecture.

\begin{con}[\cite{fas}]
$R_3(P_n)=2n-2+(n\bmod 2)$ holds for all $n \ge 1$.
\end{con}

Let $N=2n-3+(n\bmod 2)$.
It is known that $R_3(P_n) > N$ since one can find a
$(P_n,P_n,P_n;N)$-coloring for all $n \ge 1$. For $n=1,2$, these
are on an empty and one-element set of vertices, respectively.
For $n=3$ and $N=4$,
the partition of the edges of $K_4$ into 3
matchings $2K_2$ in 3 distinct colors gives a witness coloring
for $R_3(P_3)>4$.
For the general case, one can obtain a $(P_n,P_n,P_n;N)$-coloring
by using a ``blow-up" of this factorization of $K_4$ as follows.
For odd $n=2m-1 \ge 5$, a witness coloring
for $R_3(P_n)>N=4m-4$ can be obtained by blowing up each
vertex of such colored $K_4$ into a set of $m-1$ vertices,
and coloring the edges within the new 4 sets arbitrarily.
Similarly for $n=2m$, a witness coloring for $R_3(P_n)>4m-3$
can be obtained by blowing up three vertices of $K_4$ to $m-1$
vertices, and one to $m$ vertices
(for more details see \cite{GRSS}).

\medskip
It is interesting to compare (1) to the conjectured
values of three-color diagonal Ramsey numbers for cycles.

\begin{con}[\cite{erd}\cite{dzi}]
$$
R_3(C_n) = 
\left \{
\begin{array}{ll}
4n - 3 & \textrm{ for odd\ \ } n \ge 5,\\
2n & \textrm{ for even } n \ge 6.
\end{array}
\right. \eqno{(2)}
$$
\end{con}

\medskip
The odd case was conjectured by Bondy and Erd\H os in 1981 \cite{erd}, 
while the even case by the second author in 2005 \cite{dzi}. Like
with (1) for paths, (2) is known to hold for all sufficiently large $n$.
For the odd $n$ odd case, this result and an outline of the proof was
described by Kohayakawa, Simonovits and Skokan in 2005 \cite{KoSS},
and the full proof by the same authors is to appear \cite{KoSS2}.  
The case for even $n$ was settled by Benevides
and Skokan in 2009 \cite{BenSk}. These results followed the exact
asymptotic results obtained by {\L}uczak and others (see also section
6.3.1 of \cite{Rad} for details and references to other related cases).
We know that (2) holds for all $n \ge n_0$ for some $n_0$,
though there seems to be no easy way to find any concrete
upper bound on $n_0$. The first open cases of Conjecture 3
are those of $R_3(C_9)$ and $R_3(C_{10})$.

\medskip
In section 4 we will prove an interesting implication that
the even $n$ case of (2) implies the odd $(n+1)$
case of (1) for $n \ge 6$.
The equalities $R_3(C_6)=12$ \cite{yr} and $R_3(C_8)=16$
\cite{suny} were obtained with the help of computer algorithms.
Thus, it will imply that $R_3(P_7)=13$ and $R_3(P_9)=17$. 
We will also provide a computer-free proof of the latter.
Finally, we prove that $R_3(P_8)=14$, which leaves
$R_3(P_{10})$ as the first open case of (1).

\medskip
\section{Background Results}

Gy\'arf\'as, Rousseau and Schelp \cite{GRS} completely solved
the question of what is the maximum number of edges $g(m,n,k)$
in any $P_k$-free subgraph of the complete bipartite graph
$K_{m,n}$. They also characterized all the corresponding
extremal graphs. Tables III and IV in \cite{GRS} present
formulas for $g(m,n,k)$ for even and odd $k$, respectively, and
Tables I and II therein describe the constructions of
the extremal graphs achieving $g(m,n,k)$.
In our proofs of
sections 4 and 5 we will refer to these tables several times,
and hence they are reproduced in Appendix A
(with some additional comments) for convenience of the readers.

\medskip
Also in the proofs we will need some values of Tur{\'a}n numbers
for paths. In order to determine the required $T(n,P_k)$,
the following theorem by Faudree and Schelp \cite{fas},
which enhances and condenses the results by
Erd\H os and Gallai \cite{EG}, will be used.
For vertex-disjoint graphs $H_1=(V_1,E_1)$
and $H_2=(V_2,E_2)$, the join $H_1+H_2$ is a graph on
vertices $V_1 \cup V_2$ with the set of edges equal to
$E_1 \cup E_2 \cup \{ \{u,v\} \ |\  u \in V_1, v \in V_2\}$.

\bigskip
\begin{thm}[\cite{fas}\cite{EG}] 
If $G$ is a graph with $|V(G)|=kt+r$, $r<k$, $0 \le t,r$,
containing no $P_{k+1}$, then
$|E(G)|\leq t\binom{k}{2}+\binom{r}{2}$ with equality
if and only if $G$ is either $(tK_{k}) \cup K_{r}$ or
$((t-l-1)K_{k}) \cup (K_{(k-1)/2} + \overline{K}_{(k+1)/2+lk+r})$
for some $0\leq l < t$ when $k$ is odd, $t>0$,
and $r=(k \pm 1)/2$.
\label{twturan}
\end{thm}

\medskip
The following notation and terminology comes from \cite{CaV}.
For positive integers $a$ and $b$, define $r(a,b)$ as 
\begin{displaymath}
r(a,b)= a - b\Big\lfloor \frac{a}{b} \Big\rfloor = a \bmod b.
\end{displaymath} For integers $n \geq k \geq 3$, define $w(n,k)$ by
\begin{displaymath}
w(n,k)= \frac{1}{2}(n-1)k - \frac{1}{2}r(k-r-1), \eqno{(3)}
\end{displaymath} where $r=r(n-1,k-1)$.

\bigskip
\noindent
Woodall's theorem \cite{Woo} can then be formulated as follows.

\begin{thm}[\cite{CaV}]
Let $G$ be a graph on $n$ vertices and $m$ edges with $m \geq n$
and circumference $c(G)$ equal to $k$. Then
\begin{displaymath}
m \leq w(n,k),
\end{displaymath} and this result is best possible.
\label{tw01}
\end{thm}

\noindent
In \cite{CaV} (see also Appendix B), one can find the description of all
extremal graphs achieving $w(n,k)$.

\bigskip
\section {Progress on $R_3(P_{2n+1})$}

\bigskip
First we prove the following general implication.

\begin{thm}
For all $n \ge 3$, if $R_3(C_{2n})=4n$, then $R_3(P_{2n+1})=4n+1$.
\end{thm}

\medskip
\begin{proof}
The lower bound follows from the ``blow-up" construction
commented on after the statement of Conjecture 2 in
section 2 (see also \cite{GRSS}).

\medskip
For the upper bound, suppose that there exists a 3-edge
coloring of $K_{4n+1}$ without monochromatic $P_{2n+1}$.
From the assumption that $R_3(C_{2n})=4n$, we know that this
coloring contains a monochromatic $C_{2n}$
(actually, a weaker assumption $R_3(C_{2n}) \le 4n+1$
would suffice, but in Theorem 6 we preferred to stress
directly a link between Conjectures 2 and 3).
Without loss of generality, we assume that this
$C_{2n}$ is red. Now, in order to avoid
red $P_{2n+1}$ no vertex on this cycle can be connected
by a red edge to any vertex outside of the cycle. Hence,
we have a complete bipartite graph $K_{2n,2n+1}$ with only
blue and green edges. Let the parts of this bipartite graph
be called $X$ (vertices on the cycle) and $Y$ (vertices
outside of the cycle).  
Using notation of \cite{GRS}, we have
$$a=2n=|X|,\ \ \ b=2n+1=|Y|,\ \ \ c=n-1,\ \ \ a=2(c+1),$$

\medskip
\noindent
hence, if we apply the last row of Table IV
(see \cite{GRS} or Appendix A),
then we obtain that the maximum number of edges $g(2n,2n+1,2n+1)$
in any $P_{2n+1}$-free subgraph of $K_{2n,2n+1}$ is

$$g(2n,2n+1,2n+1) = (a+b-2c)c = 2n^2 + n - 3.$$

\medskip
\noindent
This implies that

$$2g(2n,2n+1,2n+1) = 4n^2 + 2n - 6 < 2n(2n+1) = |X||Y|,$$

\medskip
\noindent
and therefore blue and green edges
cannot account for all the edges of $K_{X,Y}$
without creating a monochromatic $P_{2n+1}$.
This completes the proof of the upper bound
$R_3(P_{2n+1}) \le 4n+1$.
\end{proof}

\begin{cor}{$($implied by computations for cycles$)$} 
$$R_3(P_7)=13\ \ and\ \ R_3(P_9)=17.$$
\end{cor}

\begin{proof}
The equality $R_3(P_7)=13$ was obtained previously by direct
computations \cite{yy}, while the result $R_3(P_9)=17$ is new.
It is known that $R_3(C_6)=12$ \cite{yr} and $R_3(C_8)=16$ \cite{suny},
where the upper bounds in both equalities were obtained with the help
of computer algorithms.
By Theorem 6, these imply that $R_3(P_7)=13$ and $R_3(P_9)=17$.
\end{proof}

\medskip
In the proof of the next theorem we provide a computer-free
proof of the upper bound $R_3(P_9) \le 17$. The proof of
$R_3(P_7) \le 13$ can be obtained by a similar reasoning.

\begin{thm}{$($computer-free$)$}
$$R_3(P_9) \le 17.$$
\label{tw08}
\end{thm}

\begin{proof}
We need to show that each 3-coloring of the edges of $K_{17}$
contains a monochromatic $P_9$. Let us suppose that
there is a $(P_9, P_9, P_9; 17)$-coloring $G$ with colors red, blue
and green, forming graphs $G^1$, $G^2$ and $G^3$, respectively.
Since $K_{17}$ has 136 edges, we may assume without
loss of generality that there are at least 46 red edges,
i.e. $e(G^1)\ge 46$.

Since by (3) we have $w(17,6)=46$, it follows by Theorem 5 that
$G^{1}$ contains a cycle $C_k$ for some $k \geq 6$. One can
easily verify that the critical graphs in this case
(described in \cite{CaV} and in Appendix B)
have $P_9$, and thus $k \geq 7$.
If $k \geq 9$, we immediately obtain a red $P_9$, a contradiction.
If $k=8$, then we use an argument very similar to the main
steps in the proof of Theorem 6. In order to avoid a $P_9$ in $G^{1}$
we have a bipartite graph
with partite sets of order $8$ and $9$, respectively, and
let us denote it by $G'$.
In order to avoid monochromatic $P_9$ in $G^{2}$ and $G^{3}$,
by using the last row of Table IV,
we see that $G'$ can contain at most $66$ blue and green edges.
Each of at least 6 other edges of $G'$ are red and
together with the $C_8$ they contain a red $P_9$,
again a contradiction.
Hence, in the rest of the proof we will assume that
$G$ has a red $C_7$ with vertices  $C=\{c_1, c_2, ..., c_7\}$,
the remaining vertices are $P=\{p_1, p_2, ..., p_{10}\}$,
and $G$ has no red $C_k$ for any $k \ge 8$.

\bigskip
\noindent
{\bf Claim 8A.} {\em
Let $H$ be a $(P_9,P_9,P_9;17)$-coloring, and suppose that
$|H^{1}| \ge 46$, $c(H^{1})=7$,
and $H^{1}$ contains a cycle $C=C_7$. Then there are at least $4$ vertices in
$V(H) \setminus V(C)$ joined by at least one red edge to $C$.}

\medskip
\noindent
{\bf Proof of Claim 8A.}
Consider the coloring $H$ as stated above.
Let the vertices of $C_7$ in $H^{1}$ be $C=\{c_1, c_2, ..., c_7\}$,
and the remaining vertices of $H$ are $P=\{p_1, p_2, ..., p_{10}\}$.
We will use the tables in Appendix A
several times when considering bipartite subgraphs
of $K_{|C|,k} = K_{7,k}$ for $k=10,9,8$ and $7$. We prove that
there are red edges in these bipartite subgraphs, and so for
each $k$ we obtain one more vertex in $ V(H) \setminus C$
joined to the cycle $C_7$ by at least one red edge.

The maximum possible number of edges in a bipartite graph with
partite sets 7 and $k=10$ or $9$ without $P_9$ is $7 + 3(k-1)$,
which follows from Table IV with
$a=7$, $b=k$, $c=3$ and $f_1(a,b,c)=a + (b-1)c$.
Since $2(7 + 3(k-1)) < 7k=e(K_{7,k})$ for $k=10$
and $9$, we obtain the first 2 vertices, say $p_1$ and $p_2$,
connected to $C_7$ by at least one red edge.

Now, we consider the bipartite graph
$K_{|C|,|P\setminus \{p_1,p_2\}|}=K_{7,8}$.
Similarly to the previous case, the maximum number
of edges in this bipartite graph without $P_9$ is
$7 + 3(8-1)=28$. This time, however, this is exactly
half of the edges of $K_{7,8}$, so we need to consider
the possible extremal graphs. By Table II
these extremal graphs are $G_{14}$ and $G_{15}$ with
$a=7$, $b=8$ and $c=3$, and they can be eliminated
as follows:

\begin{itemize}
\item $G_{14}=K_{7,8} - K_{4,7}$.
Clearly, $K_{4,7}$ contains a $P_9$, so it cannot consist
of edges of a single color, a contradiction.
\item $G_{15}=K_{4,4} \cup K_{3,4}$, which under bipartite
complement is isomorphic to itself.
Let us consider vertex $p_1$. To avoid red $C_8$ the vertex
$p_1$ is joined by at most $3$ red edges to the cycle $C_7$.
By considering the remaining edges from the vertex $p_1$
we see that at least one, say blue, is connected to the
$K_{4,4}$ part of $G_{15}$ in blue. This easily gives
a monochromatic $P_9$, a contradiction.
\end{itemize} 

\noindent
Thus, we have the third vertex, say $p_3$,
connected to $C$ by a red edge.

\medskip
Now we consider the bipartite graph
$K_{C,P \setminus \{p_1,p_2,p_3\}}$.
By the third case in Table IV,
the maximum possible number of edges
in this bipartite graph without $P_9$ is $7 + 3(7-1)=25$.
By Table II, there are two possible extremal graphs:
$G_{14}$ and $G_{15}$ which now are
$K_{7,7} - K_{4,6}$ and
$K_{4,4} \cup K_{3,3}$, respectively.
We proceed similarly as for $k=8$, now
by considering possible edges from $p_1, p_2$ and $p_3$
to the cycle $C_7$.
As for $k=8$, we have to consider two cases.

\begin{itemize}
\item $G_{14}$.
Suppose that $K_{7,7}-K_{4.6}$ is blue, and its bipartite complement
$K_{4,6}$ is green. In order to avoid red $C_8$, each of $p_1$ and $p_2$
may have at most 3 red edges to $C_7$, and thus at most 3 to the left
side of $K_{4,6}$.
If any of the vertices $p_1$ and $p_2$ is connected by a green edge
to $K_{4,6}$, then we have a green $P_9$. Otherwise, both
$p_1$ and $p_2$ have a blue edge to $K_{4,6}$, which easily
leads to a blue $P_9$.

\item $G_{15}$.
Suppose that $K_{4,4} \cup K_{3,3}$ is blue, and its bipartite complement
$K_{4,3} \cup K_{3,4}$ is green. In order to avoid red $C_8$, each of
$p_1$ and $p_2$ may have at most 3 red edges to $C_7$, and thus
at most 3 to the left side $L$ of $K_{4,4}$. If any of the edges
from $p_1$ or $p_2$ to $L$ is blue, then we have a blue $P_9$.
Otherwise, each of $p_1$ and $p_2$ has at least one green edge to $L$;
say $\{p_1,x\}, \{p_2,y\}$ are green for some $x,y \in L$.
If $x \not= y$ then we have a green $P_9$. Thus we can assume
that for some $x \in L$ and for $i=1,2,3$ we have: $\{p_i,x\}$
is green, the edges $\{p_i,z\}$ are red for $x \not= z \in L$,
and all edges from $p_i$ to the left side of $K_{3,3}$ are blue.
Now, in order to avoid red $C_8$ and green $P_9$, all three edges
between $p_i$'s must be blue. This leads to a blue $P_9$.
\end{itemize}

Hence, we obtain the fourth required vertex $p_4$.
This completes the proof of Claim 8A.

\bigskip
We have $m \ge 4$ vertices $M=\{p_1,\dots,p_m\} \subset P$
not on the red $C_7$ joined to it by some red edges.
Assuming that there is no red $P_9$, one can easily see
that there can be no red edges $\{p_i,p_j\}$ 
with $p_i \in M$ and $p_j \in P$. Hence, $P$
induces at most ${10-m \choose 2}$ red edges.
To avoid red $C_8$, the vertices in $M$ can be joined
by at most $3$ red edges each to $C$ (to vertices
nonadjacent on the cycle $C_7$).

First, consider the case when a vertex in $M$,
say $p_1$, has 3 red edges to $C$,
without loss of generality
$\{p_1, c_1\}$, $\{p_1, c_3\}$ and $\{p_1, c_5\}$.
Note that no vertex $p \in M$, $p \not= p_1$, can
be joined by any red edge to the vertices in the
set $\{c_2, c_4, c_6, c_7\}$, since otherwise a
red $P_9$ from $p$ to $p_1$ can be easily constructed.
In addition, if the edge $\{p_2,c_5\}$ is red, then
the edges $\{c_4, c_6\}$, $\{c_4, c_7\}$, $\{c_2, c_4\}$
are blue or green. For example, if  $\{c_4, c_6\}$ is red,
then $p_2c_5c_4c_6c_7c_1c_2c_3p_1$ is a red $P_9$. Similarly,
if $\{p_2,c_1\}$ or $\{p_2,c_3\}$ is red, then at least
three edges induced in $C$ must be blue or green.
In all cases, $C$ induces at most 18 red edges. Thus,
counting red edges in $C$, between $C$ and $P$, and in $P$,
we have
$$e(G^1) \le 18+3m+{10-m \choose 2}.\eqno{(4)}$$
Observe that the set $M\cup \{c_2, c_4, c_6, c_7\}$
induces only blue and green edges.
The latter and the known value $R(P_9,P_9)=12$
\cite{GeGy} imply that $m+4 \le 11$.
By Claim 8A we have $m \ge 4$, so $4 \le m \le 7$,
and we find that $e(G^1)<46$ for all possible $m$.
This is a contradiction.

Finally we consider the case when all vertices in $M$
are connected to $C$ by at most 2 red edges. Counting
again red edges, for all possible $4 \le m \le 10$,
we obtain
$$e(G^1) \le {7 \choose 2}+2m+{10-m \choose 2}<46,\eqno{(5)}$$
which is a contradiction.
\end{proof}

\section{Ramsey Number $R_3(P_8)$}

\bigskip
We begin with a lemma which is technically very similar
to Claim 8A within the proof of Theorem 8.

\begin{lem}
Let $H$ be a $(P_8,P_8,P_8;14)$-coloring, and suppose that
$|H^{1}| \ge 31$, $c(H^{1})=6$, and $H^{1}$ contains
a cycle $C=C_6$. Then there are at least $3$ vertices in
$V(H) \setminus V(C)$ joined by at least one red edge
to the cycle $C$.
\label{lem02}
\end{lem}

\begin{proof}
We prove this lemma similarly as Claim 8A in Theorem 8.
Consider any coloring $H$ as stated above.
Let the vertices of $C_6$ in $H^{1}$ be $C=\{c_1, c_2, ..., c_6\}$,
and the remaining vertices are $P=\{p_1, p_2, ..., p_8\}$.

We will use the tables of Appendix A several times when considering
bipartite subgraphs of $K_{|C|,k} = K_{6,k}$ for $k=8,7$ and $6$.
We prove that there are red edges in these bipartite subgraphs,
and so for each $k$ we obtain one more vertex in
$V(H) \setminus C$ joined to $C$ by at least one red edge.
By using three times Tables I and III and
considering $K_{6,k}$ for $k=8,7,6$, we can see that the maximum
number of edges in these bipartite subgraphs without $P_8$ is $3k$.
From Table I, the extremal graphs are $K_{3,l} \cup K_{3,k-l}$,
where $0 \leq l \leq k$.

First, consider the case when $k=8$. The maximum number
of edges in $K_{6,8}$ without $P_8$ is 24. Without loss of
generality consider the situation when
$K_{3,l} \cup K_{3,8-l}$, $l \geq 4$, is blue and the
bipartite complement of this graph is green. Assume
that $K_{3,l}=K_{3, |S|}$ where $S=\{p_1, p_2, ..., p_{l}\}$.
These graphs can be eliminated as follows:

\begin{itemize}
\item $l \geq 5$. To avoid a blue or green $P_8$, 
$S$ has only red edges and the vertices in $S$ can be joined
only by red edges to vertices in $P \setminus S$.
This easily gives a red $P_8$, a contradiction.    

\item $l = 4$. To avoid a blue or green $P_8$, all
the edges from $S$ to $P \setminus S$ are red.
Then we have a red $P_8$, a contradiction.
\end{itemize} 

\noindent
Thus, we have the first vertex, say $p_1$,
connected to $C$ by a red edge.

\medskip
Now, we consider the bipartite graph
$K_{|C|, |V(P)-\{p_1\}|}=K_{6,7}$.
Similarly as in the previous case, the maximum number of edges
without $P_8$ is 21. Without loss of generality,
consider the situation when
$K_{3,l} \cup K_{3,7-l}$, $l \geq 4$, is
blue and the bipartite complement of this graph is green.
Assume that $K_{3,l}=K_{3, |S|}$. To avoid a red $P_8$,
$p_1$ has only blue or green edges to the set $S$.
Then we have a blue or green $P_8$, a contradiction.
Thus, we have the second vertex, say $p_2$,
connected to $C$ by a red edge.

\medskip
Now, we consider the bipartite graph
$K_{|C|, |V(P)-\{p_1, p_2\}|}=K_{6,6}$.
Similarly as in the previous case, the maximum number of
edges in this bipartite graph without $P_8$ is 18. Without
loss of generality let us consider the situation when
$K_{3,l} \cup K_{3,6-l}$, $l \geq 3$, is blue and the
bipartite complement of this graph is green. Assume that
$K_{3,l}=K_{3, |S|}$. To avoid a red $P_8$, the edge
$\{p_1,p_2\}$ is blue or green and $p_1, p_2$ have only blue
or green edges to the set $S$. Then we have a blue or green
$P_8$, a contradiction.
Hence, we obtain the third required vertex $p_3$,
which completes the proof of Lemma 9.
\end{proof}

\begin{thm}
Three-color Ramsey number of the path $P_8$ satisfies
$$R_3(P_8)=14.$$
\end{thm}

\bigskip
\begin{proof}
We need to show that every $3$-edge coloring of $K_{14}$ contains
a monochromatic $P_8$. Let us suppose that there is a
$(P_8, P_8, P_8; 14)$-coloring $G$ with colors red, blue and green,
forming graphs $G^1$, $G^2$ and $G^3$, respectively.
Since $K_{14}$ has 91 edges, we may assume without loss of
generality that there are at least 31 red edges, i.e. $e(G^{1}) \ge 31$.

Since by (3) we have $w(14,5)=31$, it follows by Theorem 5 that
$G^{1}$ contains a cycle $C_k$ for some $k \geq 5$. One can
routinely verify that the critical graphs in this case
(described in \cite{CaV} and Appendix B) have $P_8$, and thus
$k \geq 6$. If $k \geq 8$, then we immediately
obtain a $P_8$, a contradiction.
If $k=7$, then to avoid a $P_8$ in $G^1$ we have a bipartite graph
$G'$ with two partite sets of order $7$. In order to avoid
monochromatic $P_8$ in $G^2$ and $G^3$,
by using row 3 in Table III,
we see that the graph $G'$ can contain at
most $48$ blue and green edges. At least one
remaining edge of $G'$ is red and together with the $C_7$
we have a red $P_8$, a contradiction.
Hence, in the rest of the proof we will assume that $G$
has a red $C_6$ with vertices  $C=\{c_1, c_2, ..., c_6\}$,
the remaining vertices are $P=\{p_1, p_2, ..., p_8\}$,
and $G$ has no red $C_k$ for any $k \ge 7$.

By Lemma \ref{lem02} we have
$m \ge 3$ vertices $M=\{p_1,\dots,p_m\} \subset P$
not on the red $C_6$, joined to it by some red edges.
Assuming that there is no red $P_8$,
one can easily see that there can be no red
edges $\{p_i,p_j\}$ with $p_i\in M$ and $p_j\in P$.
Hence $P$ induces at most ${8-m \choose 2}$ red edges.
To avoid red $C_7$, the vertices in $M$ can be joined
by at most 3 red edges each to $C$ (to vertices
nonadjacent on the cycle $C_6$).
We will be counting red edges in $C$, between $C$ and
$P$, and in $P$, similarly as in (4) and (5).

First, consider the case when all the vertices
in $M$ are connected to $C$ by at most 2 red edges each.
If at least one them is connected to 2 vertices in $C$,
then at least one of the edges induced by $C$ is not red,
or there are less than $2m$ edges between $C$ and $P$.
Hence, for all possible $3 \le m \le 8$, we have
$$e(G^1) \le {6 \choose 2}+2m-1+{8-m \choose 2}<31,$$
which gives a contradiction.

The remaining case is when some vertex, say $p_1$,
in $M$ is connected to $C$ by exactly 3 red edges,
and the red edges from $p_1$ to $C$ are $\{p_1, c_1\}$,
$\{p_1, c_3\}$, $\{p_1, c_5\}$.
Recall that now we can also assume that $G$
has no red $C_7$. Then no vertex
$p_i \in P$, $2 \le i \le 8$, can be joined by a red edge
to any of the vertices in the set $\{c_2, c_4, c_6\}$.
In addition, if the edge $\{p_2,c_1\}$ is red, then
the edges $\{c_2, c_4\}$, $\{c_2, c_6\}$, $\{c_4, c_6\}$
are blue or green. For example, if  $\{c_2, c_4\}$ is red,
then $p_2c_1c_2c_4c_3p_1c_5c_6$ is a red $P_8$. Similarly,
if $\{p_2,c_3\}$ or $\{p_2,c_5\}$ is red, then at least
the same three edges induced in $C$ must be blue or green.
In all cases, $C$ induces at most 12 red edges.

Observe that the set $M\cup \{c_2, c_4, c_6\}$ contains
only blue and green edges.
The latter and the known value $R(P_8,P_8)=11$
\cite{GeGy} imply that $m+3 \le 10$.
Note that if $m=7$, then the
sole vertex in $P \setminus M$ is also not in any red edge.
Therefore, $P\cup \{c_2, c_4, c_6\}$
contains only blue and green edges and, again
because of $R(P_8,P_8)=11$, there must be
a monochromatic $P_8$. Hence, we can assume that
$3 \le m \le 6$. This time we obtain
$$e(G^1) \le 12+3m+{8-m \choose 2}.\eqno{(6)}$$
Now
$e(G^1)$ can achieve 31 in (6) for $m=3$ and $m=6$,
furthermore only in cases when all (3 or 6) vertices in $M$
are connected by exactly 3 red edges to $C$. We will show
that in both cases $G$ has a blue or green $P_8$.

If $m=6$, then the equality in (6) implies that
$P\cup \{c_2, c_4, c_6\}$ contains exactly one red edge
in $P \setminus M$, or equivalently,
the $K_{11}-e$ with vertices $P\cup \{c_2, c_4, c_6\}$
has all its 54 edges blue or green. By Theorem 4 with
$k=7$, $t=1$ and $r=4$ we obtain $T(11,P_8)=27$. One
can easily check that it is not possible for two copies
of the corresponding extremal graphs to cover $K_{11}-e$.

The last situation to consider is that of $m=3$, where $G^1$
has two components: one spanned by 9 vertices of $C \cup M$
with 21 red edges and a red $K_5$ on vertices
$Q=P \setminus M =\{p_4,\dots,p_8\}$.
The set $H=M\cup \{c_2, c_4, c_6\}$ has no red edges.
Denote by $R$ the set $\{c_1, c_3, c_5\}$.
The 60 edges of $G^2\cup G^3$ form a complete $K_6$ on $H$
and a complete bipartite graph $K_{Q,H \cup R}$.
Let $P_l$ be the longest monochromatic, say blue, path in $H$,
and denote by $x$ and $y$ its endpoints. By Theorem 4 we have
$T(6,P_4)=6$, which implies that $l=6$ or $l=5$
(it is also implied by $R(P_5,P_5)=6$).
We have the following possibilities:

\bigskip
\noindent
{\bf Case 1.} There are no blue edges joining $x$ or $y$ to $Q$
(for $l=5$ or $l=6$).

\medskip
We have $H \cup R = C \cup M$, and let
$S = C \cup M \setminus \{x,y\}$. We consider the complete
bipartite graph $K_{5,7}$ with partite sets $Q$ and $S$.
Because all the edges from $x$ and $y$
to $Q$ are green, this $K_{Q,S}$ cannot have green $P_4$.
The third row of Table III with $a=5$,
$b=7$ and $c=1$, implies that there are at most
10 green edges between $Q$ and $S$.
Clearly, $K_{Q,S}$ cannot have blue $P_8$. We now
use the second row of the same Table III with $c=3$,
and see that there are at most 21 blue edges between
$Q$ and $S$. There are not enough green and blue
edges to cover all 35 edges of $K_{Q,S}$, which
is a contradiction.

\medskip
\noindent
{\bf Case 2.}
There is a blue edge from $x$ or $y$ to $Q$, say $\{x,p_4\}$, and $l=6$.

\medskip
Let the blue $P_l$ in $H$ be $xs_1s_2s_3s_4y$. If there is no blue $P_8$,
then all the edges joining $y$ to $p_i$, $5 \le i \le 8$, and joining
$p_4$ to $R$ are green. We consider the colors of the edges from
$s_4$ to the set $Q \setminus \{p_4\}=\{p_5, p_6, p_7, p_8\}$. 
This case is now broken into three subcases, as follows:

\begin{enumerate}
\item
There are at least two blue edges from $s_4$
to $Q \setminus \{p_4\}$, say
$\{s_4, p_5\}$ and $\{s_4, p_6\}$.
To avoid blue $P_8$ all the edges between
$\{p_5, p_6\}$ and $R$ must be green,
but in this case we have a green
$P_8=p_8yp_5c_1p_6c_3p_4c_5$.

\item
There is exactly one such blue edge, say $\{s_4, p_5\}$.
To avoid blue $P_8$ all the edges between
$p_5$ and $R$ must be green, but then we have a
green $P_8=c_1p_4c_3p_5yp_6s_4p_7$.

\item
All edges from $s_4$ to $\{p_5, p_6, p_7, p_8\}$ are green.
If there is a green edge between $R$ and $\{p_5, p_6, p_7, p_8\}$,
say $\{c_1,p_5\}$, then we have a green $p_7s_4p_6yp_5c_1p_4c_3$.
So, assume that all the edges from $R$ to
$\{p_5, p_6, p_7, p_8\}$ are blue.
If there is at least one blue edge from
$\{p_5, p_6, p_7, p_8\}$ to $\{s_2, s_3\}$, say $\{p_5, s_2\}$,
then we have a blue $p_4xs_1s_2p_5c_1p_6c_3$.
In the opposite case we obtain a green
$p_5s_2p_6s_3p_7s_4p_8y$.
\end{enumerate}

\medskip
\noindent
{\bf Case 3.}
There is a blue edge from $x$ to $Q$, say $\{x,p_4\}$, all the
edges from $y$ to $Q \setminus \{p_4\}$ are green, and $l=5$
(the edge $\{y,p_4\}$ can be blue or green).

\medskip
Let the blue $P_l$ in $H$ be $xs_1s_2s_3y$.
There is a vertex $z \in H \setminus \{x,y,s_1,s_2,s_3\}$,
such that the edges
$\{x,z\}$ and $\{y,z\}$ are green, since otherwise $l=6$.
This case is broken into three subcases, as follows:

\begin{enumerate}
\item
There are at least two blue edges from $p_4$ to $R$,
say  $\{p_4, c_1\}$ and $\{p_4, c_3\}$. When avoiding blue $P_8$,
we obtain a green $P_8=xzyp_8c_3p_7c_1p_6$. 

\item
There is exactly one blue edge from $p_4$ to $R$, say  $\{p_4,c_1\}$.
Then, if there is at least one green edge from $\{c_3,c_5\}$
to $Q \setminus \{p_4\}$, say $\{c_3, p_5\}$, then we have
green $P_8=acbp_8c_1p_5c_3p_4$.
In the opposite case we must have a blue complete blue
bipartite subgraph $K_{\{c_3,c_5\},\{p_5,p_6,p_7,p_8\}}$.
If there is at least one blue edge from
$\{p_5,p_6,p_7,p_8\}$ to $\{s_1,s_2,s_3\}$,
then we have a blue $P_8$, otherwise we easily find a
green $P_8$.

\item
All the edges from $p_4$ to $R$ are green.
Then, if there is at least one green edge from $R$
to $Q \setminus \{p_4\}$, say $\{c_1, p_5\}$, then
in order to avoid a green $P_8$, we must have a blue
complete bipartite $K_{\{c_3,c_5\}, \{p_6,p_7,p_8\}}$.
In the opposite case, we have a complete blue bipartite
subgraph $K_{\{c_1,c_3,c_5\},\{p_6,p_7,p_8\}}$.
If there is at least one blue edge from
$\{p_6,p_7,p_8\}$ to $\{s_1,s_2,s_3\}$,
then we have a blue $P_8$, otherwise we
have a green $P_8$.
\end{enumerate}

\medskip
\noindent
{\bf Case 4.}
There is a blue edge from $x$ to $Q$, say $\{x,p_4\}$,
there is a blue edge from $y$ to a different vertex in $Q$,
say $\{y,p_8\}$, and $l=5$.

\medskip
Let the blue $P_l$ in $H$ be $xs_1s_2s_3y$.
There is a vertex $z \in H \setminus \{x,y,s_1,s_2,s_3\}$,
such that the edges
$\{x,z\}$ and $\{y,z\}$ are green, since otherwise $l=6$.
All the edges from $\{p_4,p_8\}$ to $R \cup \{z\}$ are green.
If there are at least two green edges from a vertex
in $\{p_5,p_6,p_7\}$ to $R$, say $\{p_5,c_1\}$
and $\{p_5,c_3\}$, then we have a green $P_8$,
namely $c_5p_4c_1p_5c_3p_8zy$.
In the opposite case, we have a blue $P_4$, without
loss of generality, say $c_1p_5c_3p_6$.
To avoid a blue $P_8$, the edges $\{x,p_6\}$ and
$\{s_1,p_6\}$ are green, but then we have a green
$P_8=s_1p_6xzp_4c_1p_8c_5$.
\end{proof} 

\eject
\bigskip
It is interesting to observe that the case of
$R_3(P_8)$ required significantly more complex reasoning
than that of $R_3(P_9)$. We also tried to proceed for $P_{10}$
and $P_{11}$ similarly as in the proofs of Theorems 10 and 8,
respectively, but while the general method seems to be
applicable, the complexity of case analysis appears to
be quite harder. We did not complete these proofs.
In general, we expect that even paths
cases are harder than those for odd paths. Consequently, between
the first two open cases of Conjecture 2, namely the questions
whether it is true that $R_3(P_{10})=18$ and $R_3(P_{11})=21$,
we expect the former to be more difficult to prove.

\section*{Acknowledgements}
We are grateful to the anonymous reviewers whose detailed
comments greatly improved the proofs and general
presentation of this paper.

\input refp4.tex
\input app.tex

\end{document}

%% file: app.tex
\newpage
\section*{Appendix A}
In this appendix we present results obtained by Gy\'arf\'as,
Rousseau and Schelp \cite{GRS}.
Let $a, b, c$ be positive integers. The authors of \cite{GRS}
answered the question of what is the maximum number of edges $f_0(a,b,c)$
(resp. $f_1(a,b,c)$) in any $P_{2l}$-free (resp. $P_{2l+1}$-free),
for $l>c$, subgraph of the complete bipartite graph $K_{a,b}$ ($a\leq b$).
They also characterized all the corresponding extremal graphs.
Adjacency matrix ($a\times b$) of all extremal graphs
can be written in partitioned form as
$$G=
\begin{bmatrix}
    M_{11}       & M_{12} \\
    M_{21}       & M_{22} \\
\end{bmatrix},
$$ 
where each block in this partitioned matrix is either a matrix of
all 1's or a matrix of all 0's. Such a matrix is completely specified
by giving the size of $M_{11}$ ($s \times t$), and identifying
each $M_{ij}$ as a block of all 1's or all 0's.

Table I (resp. II) describe the constructions of all the extremal
graphs achieving $f_0(a,b,c)$ for $c \ge 1$
(resp. $f_1(a,b,c)$ for $c \ge 3$).
Table III (resp. IV) present formulas for
$f_0(a,b,c)$ for $c \ge 1$ (resp. $f_1(a,b,c)$ for $c \ge 2$).
In all four tables it is assumed that $a \le b$.

\vspace{1.5cm}

\setcounter{table}{0}

\bigskip
\begin{table}[h!]
\centering
\begin{tabular}{c|C{3cm}|C{0.9cm}|C{0.9cm}|C{0.9cm}|C{0.9cm}|C{0.9cm}|C{0.9cm}}
\hline
Graphs & Range & $s$ & $t$ & $M_{11}$ & $M_{12}$ & $M_{21}$ & $M_{22}$\\
\hline
\hline
$G_{01}$& $a\leq c$ & $a$ & $b$ & $1$ & & &  \\
\hline
$G_{02}$& $c<a \leq 2c$ & $c$ & $b$ & $1$ & & $0$ & \\
\hline
$G_{03}$& $a = 2c$ & $c$ & any & $1$ & $0$ & $0$ & $1$ \\
\hline
$G_{04}$& $a > 2c$ & $c$ & $b-c$ & $1$ & $0$ & $0$ & $1$  \\
\hline
\end{tabular}
\caption{Extremal $P_{2l}$-free subgraphs of
$K_{a,b}$, for $l > c \ge 1$ \cite{GRS}.}
\end{table}

\bigskip
\begin{table}[ht!]
\centering
\begin{tabular}{c|C{3.5cm}|C{.9cm}|C{.9cm}|C{.8cm}|C{.8cm}|C{.8cm}|C{.8cm}}
\hline
Graphs & Range & $s$ & $t$ & $M_{11}$ & $M_{12}$ & $M_{21}$ & $M_{22}$\\
\hline
\hline
&$a\leq c$ & & & & & &  \\
$G_{11}$&or & $a$ & $b$ & $1$ & & &  \\
&$a = b = c+1$ & & & & & &  \\
\hline
$G_{12}$&$a = c+1$ and \newline $b = c+2$ & $c+1$ & $c+1$ & $1$ & $0$ & & \\
\hline
$G_{13}$&$a=b$ and \newline $a=c+2$ & $c+1$ & $c+1$ & $1$ & $0$ & $0$ & $1$\\
\hline
&$b > a = c+1$ & & & & & &  \\
$G_{14}$&or & $c$ & $b-1$ & $1$ & $1$ & $0$ & $1$\\
&$c+1<a<2(c+1)$ & & & & & &  \\
\hline
&$a=2c+1$ & & & & & &  \\
$G_{15}$&or & $c+1$ & $c+1$ & $1$ & $0$ & $0$ & $1$\\
&$a = b = 2(c+1)$ & & & & & &  \\
\hline
&$b>a=2(c+1)$ & & & & & &  \\
$G_{16}$&or & $c$ & $b-c$ & $1$ & $0$ & $0$ & $1$\\
&$a > 2(c+1)$ & & & & & &  \\
\hline
\end{tabular}
\caption{Extremal $P_{2l+1}$-free subgraphs of
$K_{a,b}$, for $l > c \ge 3$ \cite{GRS}.}
\end{table}

\eject
In order to make Table II complete for $c=2$, two more
types of extremal graphs need to be defined, and
the special case of $c=1$ also can be easily handled
(see \cite{GRS}).
In this paper, for odd path we used only $c \ge 3$.
\vspace{1.5cm}

\begin{table}[h!]
\centering
\begin{tabular}{c|C{3cm}|C{3.5cm}}
\hline
No. & Range & $f_0(a,b,c)$\\
\hline
\hline
1 & $a\leq c$ & $ab$ \\
\hline
2 & $c<a<2c$ & $bc$ \\
\hline
3 & $a\geq 2c$ & $(a+b-2c)c$\\
\hline
\end{tabular}
\caption{Formulas for $f_0(a,b,c)$ for even paths, for $c \ge 1$ \cite{GRS}.}
\end{table}

\bigskip
\begin{table}[h!]
\centering
\begin{tabular}{c|C{4cm}|C{4cm}}
\hline
No. & Range & $f_1(a,b,c)$\\
\hline
\hline
1 & $a\leq c$ & $ab$ \\
\hline
2 & $a=b=c+1$ & $(c+1)^2$ \\
\hline
 & $b > a=c+1$ & \\
3 & or & $a+(b-1)c$\\
 & $c+1<a<2(c+1)$ & \\
\hline
4 & $a=b=2(c+1)$ & $2(c+1)^2$ \\
\hline
 & $b > a=2(c+1)$ & \\
5 & or & $(a+b-2c)c$\\
 & $a>2(c+1)$ & \\
\hline
\end{tabular}
\caption{Formulas for $f_1(a,b,c)$ for odd paths, for $c \ge 2$ \cite{GRS}.}
\end{table}

\vspace{1cm}
\section*{Appendix B}

In this appendix we present the characterization of all extremal
graphs achieving the bound $w(n,k)$, as described by Caccetta and
Vijayan in \cite{CaV}. The definition of $w(n,k)$ (3) with related
quantities, and Theorem~\ref{tw01} \cite{Woo,CaV} stating the bound
can be found in Section 3. These extremal graphs are needed in the
proof of our Theorem~\ref{tw08} in Section 4.

Let $G$ be any graph on $n$ vertices and $w(n,t)$ edges,
where $t=c(G)$ is the circumference of $G$, and let
$r= (n-1) \mod (t-1)$. Then $G$ is one of the graphs
described as follows.

\begin{itemize}
\item If $n=t$, then there exists one extremal graph, $K_t$,
\item If $n>t$ and $(r = t/2 - 1$ or $r = t/2)$,
then there exist two extremal graphs: 

\begin{enumerate}
\item[(1)]
$K_{t/2} + \overline{K_{n-t/2}}$, and

\item[(2)]
1-connected graph with $p-1$ cut-vertices between
consecutive blocks $B_1$, $B_2$,\dots,$B_p$,
$p= \big \lceil \frac{n-1}{t-1} \big \rceil$,
where $p-1$ blocks are equal to the graph $K_t$
and one block to the graph $K_{r+1}$,

\end{enumerate}

\item
If $n>t$, $r \neq t/2 - 1$ and $r \neq t/2$,
then there exists one extremal graph, namely
$K_1 + (\big \lfloor \frac{n-1}{t-1} \big \rfloor K_{t-1} \cup K_r)$.
\end{itemize}